\documentclass[abstracton, paper=a4, fontsize=11pt,bibliography=totoc, final]{scrartcl}
\usepackage{amsthm, amssymb, amsmath, amsfonts, mathrsfs, dsfont}

\usepackage[utf8]{inputenc}
\usepackage[T1]{fontenc}
\usepackage{lmodern}
\usepackage{graphicx,graphics}
\usepackage{mathtools}
\usepackage[english]{babel}
\usepackage{csquotes}
\usepackage{epsfig,url}  
\usepackage{bbm}
\usepackage{a4wide}
\usepackage{enumerate}
\usepackage{verbatim} 
\usepackage{color}
\usepackage{esint}
\usepackage{microtype}
\usepackage{caption}
\usepackage{subcaption}

\newcommand{\lp}{\left(}
\newcommand{\rp}{\right)}
\newcommand{\eps}{\varepsilon}
\newcommand{\Id}{\mbox{Id}}
\newcommand{\Sd}{{\mathbb{S}^{d-1}}}

\newcommand{\beqar}{\begin{eqnarray*}}
\newcommand{\eeqar}{\end{eqnarray*}}

\newcommand{\beqarl}{\begin{eqnarray}}
\newcommand{\eeqarl}{\end{eqnarray}}

\newcommand{\be}{\begin{equation}}
\newcommand{\ee}{\end{equation}}

\newcommand{\ud}{\;\mathrm{d}} 
\newcommand{\uud}{\mathrm{d}}

\providecommand{\J}{\mathcal{J}}

\providecommand{\Id}{\operatorname{Id}}

\providecommand{\eps}{\varepsilon}

\providecommand{\exp}{\operatorname{exp}}

\providecommand{\La}{\mathcal{L}}

\providecommand{\TT}{\mathbb{T}}
\providecommand{\Sp}{\mathbb{S}}
\newtheorem{theorem}{Theorem}%[section]
\newtheorem{corollary}[theorem]{Corollary}
\newtheorem{proposition}[theorem]{Proposition}
\newtheorem{definition}[theorem]{Definition}
\newtheorem{lemma}[theorem]{Lemma}

\newtheorem{remark}[theorem]{Remark}

\numberwithin{equation}{section}
\numberwithin{theorem}{section}
\DeclareMathOperator{\sign}{\mathrm{sign}} 

\DeclareMathOperator*{\esssup}{ess\,sup}

\newcommand{\cc}{\mathfrak c}

\newcommand{\Reals}{{\mathbb R}}
\newcommand{\R}{{\mathbb R}}

\newcommand{\T}{{\mathbb T}}

\renewcommand{\varrho}{{\rho}}

\newcommand{\ep}{\varepsilon}

\newcommand{\intall}{\int_{\TT^d\times \Sd}}

\newcommand{\me}{\mathcal{S}}

\usepackage[margin=2cm,includefoot]{geometry}

\mathtoolsset{showonlyrefs}

\title{Stability of equilibria of the spatially inhomogeneous Vicsek-BGK equation across a bifurcation}
\author{Sara Merino-Aceituno\footnote{Faculty of Mathematics, University of Vienna, Oskar Morgenstern-Platz 1, 1090 Wien, Austria}, Christian Schmeiser$^*$, Raphael Winter\footnote{School of Mathematics, Cardiff University, Abacws, Senghennydd Road, Cathays, Cardiff CF24 4AG,
UK} }
\begin{document}

\maketitle

\begin{abstract}
The Vicsek-BGK equation is a kinetic model for alignment of particles moving with constant speed between stochastic reorientation events with sampling from a von Mises
distribution. The spatially homogeneous model shows a steady state bifurcation with exchange of stability. The main result of this work is an extension of the bifurcation result
to the spatially inhomogeneous problem under the additional assumption of a sufficiently large Knudsen number.
The mathematical core is
the proof of linearized stability, which employs a new hypocoercivity approach based on Laplace-Fourier transformation. The bifurcation result includes global existence of
smooth solutions for close-to-equilibrium initial data. For large data smooth solutions might blow up in finite time whereas weak solutions with bounded Boltzmann entropy are
shown to exist globally.
\end{abstract}

%\tableofcontents

\paragraph{Acknowledgment:} The authors thank Anton Arnold, Stefan Egger and Amic Frouvelle for fruitful discussions and corrections on the preprint version. This work has been supported by the Austrian Science Fund, grant no. F65.  The work of SMA was supported  by the Vienna Science  and  Technology  Fund  (WWTF)  [10.47379/VRG17014].

\section{Introduction}

\paragraph{The Vicsek-BGK equation} is a kinetic model for alignment dynamics. It describes the time-evolution of a particle distribution in position and direction of motion. 
%Investigating the long-time behaviour of kinetic equations is a classical research theme, and many methods have been developed for this, like energy or hypocoercivity-based methods. The model considered here, however, is structurally different from previously studied models since the number of conservation laws is smaller than the dimension of the equilibrium set and it presents a steady state bifurcation with exchange of stability.

%While known methods cannot be applied to investigate the long-time behaviour of this equation, the structure of the equation is such that it allows to directly compute many of its properties. In the following, we give more details on the equation and its challenges.

%\subsection*{The Vicsek-BGK equation}

Let $\T^d$ be the (flat) $d$-dimensional torus and $\mathbb{S}^{d-1}$ the unit sphere in $\R^d$, $d\ge 2$. The initial value problem for the Vicsek-BGK equation for the particle density $F=F(t,x,\omega)$, depending on position $x\in\T^d$, movement direction $\omega\in \Sp^{d-1}$, and time $t\geq 0$, is given by
\begin{equation}\label{nonlinearBGK}
    \begin{aligned} 
    \partial_t F +  \gamma \omega \cdot \nabla_x F &= \rho_F M_{J_F} - F \,, \\
        F(0,x,\omega) &= F^\circ(x,\omega) \,,
\end{aligned}
\end{equation}
where $\gamma>0$ is a given parameter and $F^\circ\ge 0$ denotes the initial datum. The position density $\rho_F$, the flux $J_F$, and the von Mises distribution $M_J$ are given by
\begin{align}
    \rho_F(t,x)   &:= \int_{\Sd} F(t,x,\omega) \ud{\omega} \,, \\
    J_F(t,x)      &:= \int_{\Sd} \omega F(t,x,w) \ud{\omega} \,,\\
    M_J(w)&:= \frac{\exp(\omega \cdot J)}{Z(J)} \,, \qquad\mbox{with } Z(J):= \int_{\Sd} \exp(\omega \cdot J)\, \ud{\omega} \,. \label{def:MJ}
\end{align}
The left hand side of the Vicsek-BGK equation \eqref{nonlinearBGK} is a transport term describing movement with constant speed in the direction $\omega$. The right-hand-side, 
the 'collision operator' in the language of kinetic theory, describes a direction jump process, where post-collisional directions are sampled from the von Mises distribution.

The problem has been nondimensionalized, where the scaling of $F$ has been chosen to remove a (temperature) parameter in the exponent of the von Mises distribution, the
reference length is the length of the original torus, and the reference time is the mean time between directional jumps. The dimensionless parameter $\gamma>0$ is the Knudsen number, i.e. the ratio between the mean free path and the length of the torus or, equivalently, between the free flight time and the time a particle needs to traverse the torus. A second important dimensionless parameter is the scaled total number of particles or, equivalently, the scaled total mass $\mu$, which is (at least formally) conserved by the 
evolution:
\begin{align}
    \int_{\TT^d}\int_{\Sd} F(t,x,\omega) \ud{\omega} \ud{x} = \int_{\TT^d}\int_{\Sd} F^\circ(x,\omega) \ud{\omega} \ud{x} =: \mu \,,\qquad t\ge 0\,. 
\end{align}
The classical prototypes for individual based alignment models are the Cucker-Smale model \cite{cucker2007emergent}, where particle speeds can be arbitrary, and the Vicsek model 
\cite{vicsek1995novel}, where all
particles share the same speed and vary only their movement direction. In both models the velocity/direction of each particle is changed towards a weighted average over 
the other particles. A link between both types of models is provided by Cucker-Smale models with additional relaxation of the speed towards a constant value (see, e.g., \cite{Aceves2019}). Kinetic versions of these models can be seen as models (sometimes limits \cite{degond2008continuum}) for large particle ensembles. The Vicsek-BGK equation \eqref{nonlinearBGK} belongs to 
this group with the collision operator mimicking alignment under stochastic perturbations. Strongly related is the Vicsek-Fokker-Planck version of the model 
\cite{degond2013macroscopic,degond2015phase,frouvelle2012continuum,frouvelle2012dynamics}, where the collision operator is replaced by a Fokker-Planck type differential 
operator in $\omega$, which has the same equilibria as \eqref{nonlinearBGK}. Extensions of these models, used for the description of flocking dynamics, typically include 
long range attraction and short range repulsion via distance dependent interaction potentials \cite{cao2020asymptotic,carrillo2017review,carrillo2010particle,cucker2007emergent}.

\paragraph{The spatially homogeneous problem -- bifurcation:}
\label{sec:bifurcations}

For given total mass $\mu\ge 0$ the Vicsek-BGK equation \eqref{nonlinearBGK} admits spatially homogeneous equilibria of the form 
\begin{align} \label{eq:equilibria_homogeneous}
    F_\infty(x,\omega) = \mu M_{J_\infty}(\omega) \,,     
\end{align}
where $J_\infty\in \R^d$ must satisfy the consistency relation
\begin{align} \label{eq:consistency} 
    J_\infty= \mu\int_{\Sd} \omega M_{J_\infty}(\omega) \ud \omega = \mu \cc(|J_\infty|)\frac{J_\infty}{|J_\infty|} \,, \qquad \mathfrak{c}(r) := \frac{\int^\pi_{0} \cos\theta e^{r\cos\theta} \sin^{d-2}\theta \ud \theta}{\int^\pi_{0}  e^{r\cos\theta} \sin^{d-2}\theta \ud \theta} \,.
\end{align}
Obviously, the direction of $J_\infty$ is arbitrary, but its modulus $L=|J_\infty|$ has to satisfy 
\begin{equation}\label{eq:consistency_relation2}
   L=\mu \cc(L) \,.
\end{equation}
This equation has been studied in \cite{degond2013macroscopic,degond2015phase,frouvelle2012continuum,frouvelle2012dynamics}. For any $\mu\ge 0$ it has the trivial
solution $L=0$, which is the only solution for $\mu\le d$. At the critical mass $\mu=d$ a bifurcation occurs, and for $\mu>d$ there exists a continuous branch of unique nontrivial solutions $L_\mu\in(0,\mu)$ with $L_d=0$. At the bifurcation point we expect an exchange of stability from the trivial to the nontrivial steady state.
Motivated by this, we define the manifold of supposedly stable steady states:

\begin{definition}[Equilibrium set]
\label{def:Steady}
The equilibrium set $\me(\mu) \subset \mathbb R^{d}$ is defined by 
\begin{equation} \label{eq:SMu}
  \me(\mu) := \left\{ \begin{array}{ll} \{0\} \,,& 0\le\mu\le d \,,\\      \{ L_\mu\omega :\, \omega\in \Sd\} \,,& \mu>d \,,\end{array}\right.
\end{equation}
where $L_\mu>0$ denotes the nontrivial solution of \eqref{eq:consistency_relation2}. 
\end{definition}

%\begin{figure}
%     \centering
%     \begin{subfigure}[b]{0.4\textwidth}
%         \centering
%         \includegraphics[width=\textwidth]{coeffc.png}
%         \caption{Plot of the order parameter $\cc=\cc(r)$. }
%         \label{fig:c}
%     \end{subfigure}
%     \hfill
%     \begin{subfigure}[b]{0.4\textwidth}
%         \centering
%         \includegraphics[width=\textwidth]{coeff_kappaintoc.png}
%         \caption{Plot of the quotient $r/\cc(r)$. }
%         \label{fig:kappaintoc}
%     \end{subfigure}
%        \caption{Plots of $\cc=\cc(r)$ and $r/\cc(r)$ in dimension 3. Graphically we see that there is a unique $\rho$ such that $\rho=r/\cc(r)=\rho$ for $\rho\geq d=3$. The analytical prove of this can be found in \cite{}.}
%        \label{fig:coeff}
%\end{figure}

For the spatially homogeneous problem our expectation is correct:
\begin{proposition}[Convergence to equilibria in the spatially homogeneous case]
\label{prop:space-homogeneous-stability}
The problem
\begin{align}\label{space-hom}
    \partial_t F = \rho_F M_{J_F} - F \,, \quad F(0,\cdot)=F^\circ(\cdot) \,,
\end{align}
with $F^\circ \in L_+^1(\Sd)$ has a global solution satisfying $ \lim_{t\to\infty} F(t,\omega) = \mu M_{J_\infty}(\omega)$ with
$$
   \mu= \int_{ \Sd} F^\circ(\omega) \ud\omega \,,\qquad J_\infty = \left\{ \begin{array}{ll} 0 \,,& \mbox{for } 0\le\mu\le d \mbox{ or } J_{F^\circ}=0\,,\\
                                                                       L_\mu \frac{J_{F^\circ}}{|J_{F^\circ}|} \,,& \mbox{for } \mu>d \mbox{ and } J_{F^\circ} \ne 0\,. \end{array}\right.
$$
\end{proposition}

\begin{proof}
By mass conservation we have $\rho_F(t) = \mu$, $t\ge 0$. The flux then satisfies the closed initial value problem
$$
   \frac{\ud}{\ud t} J_F = \left( \frac{\mu\cc(|J_F|)}{|J_F|} - 1\right)J_F \,,\qquad J_F(0) = J_{F^\circ}\,,
$$
which conserves the direction of $J_F$. Its solution has the form $J_F(t) = L(t)\frac{J_{F^\circ}}{|J_{F^\circ}|}$, where $L(t)$ solves
$$
   \frac{\ud L}{\ud t} = \mu\cc(L) - L \,,\qquad L(0) = |J_{F^\circ}| \,.
$$
The function $\cc$ has the properties
$$
   \cc(0)=0 \,,\quad \cc'(0) = \frac{1}{d} \,,\quad L\mapsto \frac{\cc(L)}{L} \mbox{ is strictly decreasing and tends to zero as } L\to\infty \,,
$$
which have been shown in \cite{frouvelle2012dynamics}. They imply that 
$$
   \lim_{t\to\infty} L(t) = \left\{ \begin{array}{ll} 0 \,,& 0\le\mu\le d  \mbox{ or } J_{F^\circ}=0\,,\\ L_\mu \,,& \mu > d \mbox{ and } J_{F^\circ} \ne 0\,. \end{array}\right.
$$
The result is now obvious by solving \eqref{space-hom} with given $\rho_F$ and $J_F$.
\end{proof}

The situation is simpler in models, where the equilibrium is a von Mises distribution with $J_F$ replaced by its normalized version. In this case the consistency relation corresponding to \eqref{eq:consistency} is always satisfied and there is only one stable equilibrium without any bifurcation behavior (see, e.g., \cite{kang2023dynamics}).

\paragraph{Challenges -- results -- perspectives:}
It is the main goal of this work to extend the result of Proposition \ref{prop:space-homogeneous-stability} to the spatially inhomogeneous model \eqref{nonlinearBGK}.
Since the transport operator is not dissipative, a hypocoercivity result \cite{dric2009hypocoercivity} is needed to prove that the combination of collisions and transport provide dissipation also in the position variable. Challenges are caused by three somewhat related properties of the model:
\begin{itemize} 
\item The bifurcation behavior discussed above, 
\item the lack of a sufficient number of conservation laws to explicitly determine the asymptotic equilibrium from the initial data, and finally 
\item the lack of an entropy. 
\end{itemize}
Since the established approaches
for proving hypocoercivity \cite{achleitner2016linear,achleitner2023hypocoercivity,dolbeault2015hypocoercivity,duan2011hypocoercivity,herau2006hypocoercivity,dric2009hypocoercivity} are based on entropy
dissipation, they cannot be applied here.
The lack of a complete set of conservation laws also has important implications for the derivation of macroscopic limits since, again, classical methods fail. The curious reader is referred to \cite{degond2008continuum}.

Our analysis starts with a rather straightforward approach for the linearized equation, employing a Laplace-Fourier transformed problem. This decouples the Fourier modes,
and therefore separates an explicitly treatable low dimensional problem showing the bifurcation behavior, from the higher order modes which are expected to decay. This can actually be proven by rather involved estimates and subsequent inverse Laplace transformation, which seems to be a new approach to showing hypocoercivity.  The decay result in $L^2$ can easily be raised to an $L^2$ based Sobolev space of high enough differentiability order to allow the 
control of the nonlinearities. Since the equilibrium is not known, an iterated procedure with linearizations around improved approximations of the equilibrium is necessary.
The final result is exponential convergence to equilibrium for initial data close enough to the set of stable equilibria.

The approach also produces global existence of strong solutions for initial data close to the set of stable equilibria. For large initial data finite time blow-up of strong solutions 
cannot be excluded. Therefore we complement our analysis by proving global existence of weak solutions for large initial data with bounded Boltzmann entropy. Although the
entropy is not deceasing in general, it can be shown to increase at an at most exponential rate.

Many open questions and venues of research remain:
\begin{itemize} 
\item The domain of attraction of the stable equilibria considered here is unclear. Global attraction might be conjectured, but so far we do not even have boundedness of large solutions.
\item Supposedly our results can be extended to different settings of the position space, where hypocoercivity approaches have been used successfully, such as whole space 
with \cite{dolbeault2015hypocoercivity} or without \cite{bouin2020hypocoercivity} confinement.
\item 
One can also consider other versions of the Vicsek model, in particular, the Vicsek-Fokker-Planck studied in \cite{degond2013macroscopic,degond2015phase,frouvelle2012continuum,frouvelle2012dynamics}. There exists already some preliminary results on the stability of equilibria in the space-inhomogenous case presented in \cite{frouvelle}, though the methodology there is very different from the one used here.
\item
Finally, there are other models for collective dynamics that present phase transitions, including models based on body attitude coordination and apolar alignment \cite{degond2020phase,degond2021kinetic,ginelli2010large}. For these models the question of stability in the space-inhomogeneous case is still a completely unexplored subject.
\end{itemize}

The rest of this work is structured as follows: Section~\ref{sec:main_results} contains precise formulations of the main results. The results on the linearized problem are proven 
in Section \ref{sec:lin}. Local stability for the nonlinear problem across the bifurcation is shown in Section \ref{sec:convergenceProof}. The global existence result of weak solutions
with bounded Boltzmann entropy is proven in Section \ref{sec:existence}. Finally there are two short appendices with technical results.

\section{Main results}
\label{sec:main_results}

\subsection*{Stability of the linearized equation}
\label{sec:results_linearized}

The linearization of \eqref{nonlinearBGK} around a spatially-homogeneous equilibrium $F_\infty=\mu  M_{\J}$ with $\J \in \mathcal{S}(\mu)$ is given by
\begin{equation}\label{eq:linearizedBGK}
    \begin{aligned} 
    &\partial_t f + \gamma \omega \cdot \nabla_x f =  \rho_f M_{\J} + \mu  J_f\cdot  \nabla_J M_{\J}- f \,, \\
    &f(t=0) =f^\circ, \qquad \mbox{ with }\intall f^\circ(x,\omega) \ud x \ud \omega =0 \,.
\end{aligned}
\end{equation}
The constraint on the initial datum $f^\circ$ of the perturbation $f$ is a consequence of the assumption that $\mu$ is the (conserved) total mass of the initial datum 
$F^\circ$ of $F$. Mass conservation carries over to the linearized problem, and therefore we expect
\begin{align}
    \intall f(t,x,\omega) \ud x \ud \omega= 0 \,,\qquad t\ge 0 \,.
\end{align}
This can be written as $\overline{ \rho_f}(t) = 0$, $t\ge 0$, using the spatial average
\begin{align} \label{eq:def_bracket_operator}
    \overline G(t,\omega) := \int_{\T^d}  G(t,x,\omega) \ud{x} \,. 
\end{align}
Contrary to the nonlinear problem, the equilibrium for the linearized equation can be determined from the initial data by additional conservation laws.
The spatial average of the flux satisfies the closed equation
\begin{equation}\label{Jf-ODE}
    \frac{\ud}{\ud t}\overline{J_f} = \left( \mu \int_\Sd \omega\otimes\nabla_J M_{\J}\ud\omega - \Id\right) \overline{J_f} \,.
\end{equation}
We shall prove (Lemma \ref{lem:CJ}) that the coefficient matrix is symmetric and that for $\mu<d$ (i.e. $\J=0$) it has negative eigenvalues and, thus, 
$\overline{J_f}(t)\to 0$ as $t\to\infty$. On the other hand for $\mu>d$ the coefficient matrix possesses
the $(d-1)$-dimensional nullspace $\J^\bot$ and the remaining eigenvalue with eigenvector $\J$ is negative for $\mu$ close to $d$. 

When the solution of \eqref{eq:linearizedBGK} converges to a steady state, we expect hypocoercivity, i.e. spatial homogeneity of the steady state and therefore
$$
   \rho_f(t,x) \approx \overline{\rho_f}(t) = 0 \,,\quad J_f(x,t) \approx \overline{J_f}(t) \,,\qquad\mbox{as } t\to\infty \,.
$$
As a consequence, the long time behaviour of the solution of \eqref{Jf-ODE} completely determines the equilibrium. In particular, we define
\begin{equation}\label{f_infty}
   f_\infty(x,\omega) = \left\{ \begin{array}{ll} 0 \,, & \mbox{for } \mu \le d \,,\\ 
                                                                     \mu \left(P_\J^\bot \overline{J_{f^\circ}}\right)\cdot \nabla_J M_\J(\omega) \,,& \mbox{for } \mu>d \,,
                                            \end{array}\right.
\end{equation}
with the orthogonal projection
\begin{align} \label{eq:operatorProjection}
    P^\perp_\J:= \Id -\frac{\J}{|\J|}\otimes \frac{\J}{|\J|} \,,
\end{align} 
to the nullspace of the coefficient matrix in \eqref{Jf-ODE}. 

We remark that for $\mu\leq d$, the linearized equation reads
\begin{equation}\label{eq:linearizedBGKstable}
	\begin{aligned} 
		&\partial_t f + \gamma \omega \cdot \nabla_x f =  \rho_f M_{0} + \mu  J_f\cdot  \omega M_{0}- f \,, \\
		&f(t=0) =f^\circ, \qquad \mbox{ with }\intall f^\circ(x,\omega) \ud x \ud \omega =0 \,.
	\end{aligned}
\end{equation}
In this case, any function $f\in L^2(\mathbb{T}^d\times \Sd)$ can decomposed orthogonally into
\begin{align*}
	f = \rho M_0 + d J_f\cdot \omega  M_0 + f^\perp,
\end{align*}
and we obtain
\begin{align} \label{eq:L2stab}
	\frac12 \partial_t \|f\|^2_{L^2(\mathbb{T}^d\times \Sd)} = -\|f^\perp\|^2_{L^2(\mathbb{T}^d\times \Sd)} - (d-\mu) \int_{\mathbb{T}^d} |J(x)|^2 \ud{x} \leq 0, 
\end{align}
that the $L^2$ norm is monotone decreasing for $\mu \leq d$.

\begin{theorem}[Exponential stability for the linearized equation]\label{th:spec-stab}
There exist $\gamma_{min}, \kappa>0$, such that for $\gamma\ge \gamma_{min}$, $0<\mu\le d+\kappa$, $\mu\ne d$, $m\geq 0$,  there exist $\lambda, C>0$,
such that for each $\J \in \mathcal S(\mu)$ the solution $f$ of \eqref{eq:linearizedBGK} satisfies
    \begin{align} \label{eq:decay_linear}
        \|f(t,\cdot,\cdot)- f_\infty\|_{H^m_x(\T^d \times \Sd)}\leq C e^{-\lambda t}\|f^\circ\|_{H^m_x(\T^d \times \Sd)}  \,,
    \end{align}
where the space $H^m_x$ is defined by
$$
 \|g\|_{H^m_x(\T^d \times \Sd)} := \sum_{|\alpha|\leq m}  \|D^\alpha_x g\|_{L^2(\TT^d \times \Sp^{d-1})}  \,.
$$
The decay rate satisfies $\lambda = O(|\mu-d|)$ as $\mu\to d$.
\end{theorem}

\begin{remark}
    In fact, the $L^2$ stability estimate~\eqref{eq:L2stab} could be used to extend this result to arbitrary $\gamma>0$, however without constructive bounds on $\lambda>0$.
\end{remark}

\subsection*{Exponential stability of equilibria for the nonlinear equation}
\label{sec:convergence}

We start with local existence of strong solutions of the nonlinear problem and provide a blow-up criterion. The proof follows standard arguments and will only be outlined in Section~\ref{sec:convergenceProof}. 
\begin{lemma}[Existence] \label{lem:existence}
    Let $m> d/2$ and $F^\circ\in H^m_x(\TT^{d}\times \Sp^{d-1})$. Then there exists a time  $T_*>0$ such that \eqref{nonlinearBGK} has a strong solution $F\in C^1([0,T_*);H^m_x(\TT^{d}\times \Sp^{d-1}))$ and either
    \begin{align}
        T_* = \infty \,, \qquad  \text{or}\qquad
        \limsup_{t\rightarrow T_*} \|\rho(t)\|_{L^\infty(\TT^d)} = \infty \,. 
    \end{align}
    Furthermore, the total mass of $F^\circ$ is conserved by the evolution:
    \begin{align} \label{eq:totalmass} 
        \int_{\TT^d \times \Sp^{d-1}} F(t,x,\omega) \ud{\omega} \ud{x} = \int_{\TT^d \times \Sp^{d-1}} F^\circ(x,\omega) \ud{\omega} \ud{x} =: \mu \,,\qquad 0\le t < T_*\,.
    \end{align}
\end{lemma}

The following theorem is the main result of the paper and includes the global well-posedness and local stability of von Mises equilibria, extending the spectral stability 
result Theorem \ref{th:spec-stab}. 

\begin{theorem}[Long time behaviour - bifurcation for the Vicsek-BGK equation \eqref{nonlinearBGK}]
\label{th:convergence}
Let  $m> d/2$ and $\gamma,\mu$ as in Theorem~\ref{th:spec-stab}.
Then there exists $\eps_0>0$ such that for any $J_1 \in \mathcal{S}(\mu)$ and any non-negative initial datum $F^\circ\in H^m_x(\TT^{d}\times \Sp^{d-1})$
with $\mu=\int_{\TT^d \times \Sd} F^\circ(x, \omega)  \ud{\omega}\uud{x}$ and 
\begin{align} \label{eq:lem_assump2}
        \|F^\circ-F_{\mu,J_1}\|_{H^m_x(\TT^{d}\times \Sp^{d-1})}< \eps_0 \,, 
    \end{align}
there exists a global solution $F=F(t)$ of~\eqref{nonlinearBGK} with initial datum $F^\circ$. Moreover, there exist  $\hat\lambda,C>0$ and $J_\infty \in \me(\mu)$  such that
     \begin{align} \label{eq:conv_equilibrium}
    \|F(t) - F_{\mu,J_\infty}\|_{H^m_x(\TT^{d}\times\Sp^{d-1})}\leq C e^{-\hat\lambda t}\|F^\circ-F_{\mu,J_1}\|_{H^m_x(\TT^{d}\times \Sp^{d-1})} \,.
    \end{align}
\end{theorem} 

\begin{remark} a) The proof, given in  Section~\ref{sec:convergenceProof}, shows that the result degenerates as $\mu\to d$, whence both 
$\eps_0>0$ and $\hat\lambda$ tend to zero, just as the exponential decay rate $\lambda$ in Theorem \ref{th:spec-stab}. \\
b) By the lack of an appropriate conservation law there is no formula for the equilibrium flux $J_\infty$. The proof shows that it satisfies
$$
    |J_1-J_\infty| \le C \|F^\circ-F_{\mu,J_1}\|_{H^m_x(\TT^{d}\times \Sp^{d-1})}  \,.
$$
\end{remark}

\subsection*{Global existence of weak solutions}
\label{sec:existenceGlobal}
Our stability result close to equilibrium is complemented by a global-in-time well-posedness for initial data with finite entropy. 
\begin{theorem}
    \label{th:existence_solution}
    Let $d\geq 2$, $0\le F^\circ \in L^1(\TT^d \times \Sd)$, and 
    \begin{align}
        \int_{\TT^d \times \Sd} F^\circ |\log F^\circ| \ud{x} \ud{\omega} = E_0 < \infty \,.
    \end{align}
 Then there exists a nonnegative mild solution $F\in C([0,\infty), L^1(\TT^d\times \Sd))$ of the Vicsek-BGK problem~\eqref{nonlinearBGK}. Moreover, the total mass of $F$ 
 is constant in time and the entropy grows at most exponentially, i.e. there exist $c,C>0$ such that
    \begin{align} \label{eq:entropybound} 
        \int_{\TT^d \times \Sd} F(t,x,v) |\log F(t,x,v)| \ud{x} \ud{v} \leq C (1+ e^{ct})\,.  
    \end{align}
\end{theorem}
\begin{remark}
    There is no satisfactory theory of uniqueness known to us for entropy solutions such as those provided by Theorem~\ref{th:existence_solution}. 
\end{remark}

\section{The linearized equation}\label{sec:lin}

\subsection*{Spectral stability -- proof of Theorem \ref{th:spec-stab}}

Well posedness of the linearized problem \eqref{eq:linearizedBGK} in $L^2(\TT^d \times \Sd)$ is a simple consequence of the fact, that the right hand side 
of the equation is a bounded operator, which follows from the inequalities
$$
    |\rho_f|, |J_f| \le \sqrt{|\Sd|}\, \|f\|_{L^2(\TT^d)}  \,.
$$
For the stability analysis we shall employ the Fourier-Laplace transform: For a function $g=g(t,x)$, $t\ge 0$, $x\in \TT^d$, its Fourier transform is defined by
\begin{equation}
    \label{eq:def_Fourier_transform}
\hat g_k(t) := \int_{\TT^d} g(t,x) e^{-ix\cdot k/\gamma } \ud x \,, \qquad k\in \gamma \mathbb Z^d :=\{\xi \in \mathbb{R}^d: \xi=\gamma k', k'\in \mathbb{Z}^d\} \,,
\end{equation}
and $g$ can be recovered by the inverse transform 
\begin{align*}
    g(t,x) = (2\pi)^d \sum_{k\in  \gamma \mathbb Z^d} \hat g_k(t) e^{ix\cdot k/\gamma} \,.
\end{align*}
The Fourier-Laplace (F-L) transform of $g$ is defined by
\begin{equation}
    \label{eq:def_LF_transform}
\tilde g_k(z) := \int^\infty_0 \hat g_k(t) e^{-zt} \ud t \,, \qquad z\in\mathbb C \,,\quad k\in \gamma\mathbb Z^d\,.
\end{equation}
Applying the F-L transform to the linearized equation \eqref{eq:linearizedBGK}, we obtain 
\begin{align}\label{eq:F-L}
    (1+ z+i k\cdot \omega ) \tilde{f}_k(z,\omega)  = \tilde{\rho}_k(z) M_\J + \mu \tilde{J}_k(z)\cdot \nabla_J M_\J + \hat{f}_k^\circ(\omega) \,,
\end{align}
where, for simplicity, we abbreviate by $\tilde \rho = \tilde \rho_f$ and $\tilde J=\tilde J_f$. By division by the coefficient on the left hand side and taking moments with
respect to $\omega$ we obtain closed systems for $(\tilde J_k,\tilde\rho_k)$ for each $k\in\gamma\mathbb Z^d$:
\begin{equation} \label{eq:tilde}\begin{aligned}
    \tilde{J}_k  &= b_\J \tilde{\rho}_k   + \mu A_\J \tilde{J}_k + r_{J,k} \,,\\
    \tilde{\rho}_k &= a_\J \tilde{\rho}_k  + \mu \,\bar b_\J  \cdot \tilde{J}_k + r_{\rho,k} \,. 
\end{aligned}\end{equation} 
with the coefficients
\begin{equation}\label{aJ-etal}\begin{aligned}
   a_\J(z,k) &= \int_\Sd \frac{M_\J} {1+z+ik\cdot \omega } d\omega \,,\qquad A_\J(z,k) = \int_\Sd \frac{  \omega \otimes \nabla_J M_\J }{1+z+ik\cdot \omega}d\omega \,,\\
   b_\J(z,k) &= \int_\Sd \frac{\omega M_\J}{1+z+ik\cdot \omega} d\omega \,,\qquad \bar b_\J(z,k) = \int_\Sd \frac{\nabla_J M_\J}{1+z+ik\cdot \omega} d\omega \,,
\end{aligned}\end{equation}
and with the inhomogeneities
\begin{align}\label{r_k}
    r_{J,k} =  \int_\Sd \frac{\omega \hat{f}_k^\circ }{1+z+i k\cdot \omega} d\omega \,,\qquad r_{\rho,k} = \int_\Sd \frac{\hat{f}_k^\circ }{1+z+i k\cdot \omega} d\omega  \,.
\end{align}
With the formula \eqref{grad-MJ} for $\nabla_J M_\J$ we have the following relations between the coefficients:
$$
   A_\J = \int_\Sd \frac{  \omega \otimes \omega M_\J }{1+z+ik\cdot \omega}d\omega - \frac{1}{\mu} b_\J\otimes\J \,,\qquad
   \bar b_\J = b_\J - \frac{1}{\mu} a_\J \J \,.
$$
The problem will be solved after reduction to a scalar equation by elimination of $\tilde J_k$. This requires the following result, where initially we restrict our attention to $k\ne 0$.

\begin{lemma} \label{lem:stabnontriv}
    For any $\gamma>0$ and $d\ge 2$ there exist constants $\kappa, \delta, C >0$, such that for any $\mu\in [0,d+\kappa)$, $\J\in\mathcal S(\mu)$, $z\in\mathbb C$
    with $\Re(z)\geq -\delta$, and $0\ne k\in \gamma\mathbb Z^d$ the matrix $\Id - \mu A_\J(z,k)$ is invertible (where `$\Id$' denotes the identity matrix) and satisfies
    \begin{align} \label{eq:Astable}
        \|(\Id - \mu A_\J(z,k))^{-1}\| \leq C \,.
    \end{align}
where $\|\cdot\|$ is any matrix norm.
\end{lemma}
\begin{proof} Since the matrix depends smoothly on $\mu$ and $\J$, it is sufficient to prove the result for the case $\mu\le d$, $\J=0$, since then it can be extended
by continuity to $0<\mu-d\ll 1$, $|\J|\ll 1$.

With $\Omega\in \mathbb{C}^d$, $|\Omega|=1$, it is sufficient to prove that $\Omega^{tr}(\Id - \mu A_0(z,k))\Omega$ is bounded away from zero by a constant only
depending on $\gamma$ and $d$ (where `$tr$' denotes `transpose'). We immediately have
\begin{align}
  \left|\Omega^*(\Id - \mu A_0)\Omega \right| \ge 1 - \mu \int_\Sd \frac{|\omega\cdot\Omega|^2 M_0}{|1+z+ik\cdot\omega|}\ud\omega
  \ge 1 - d\int_\Sd \frac{|\omega\cdot\Omega|^2 M_0}{((1-\delta)^2 + (\Im(z)+k\cdot\omega)^2)^{1/2}}\ud\omega \,.
\end{align}
As a consequence of
$$
    \int_\Sd \omega_i \omega_j M_0 \ud\omega = \frac{\delta_{ij}}{d} \,,\qquad i,j = 1,\ldots d \,,
$$
the right hand side vanishes, if the denominator is replaced by 1. Therefore help from the imaginary part is needed to allow
for a positive $\delta$. We define
\begin{equation}\label{def:Sk}
    \mathbb S_k = \{\omega\in \mathbb S^d:\, \sign(k\cdot\omega) = \sign(\Im(z))\,,\, (k\cdot\omega)^2 \ge |k|^2/2 \} \,,
\end{equation}
implying
$$
    (\Im(z) + k\cdot\omega)^2 \ge \frac{\gamma^2}{2} \qquad \mbox{for } \omega\in\mathbb S_k \,,
$$
where we have used $0\ne k \in \gamma\mathbb Z^d$. It is important to note that $\mathbb S_k$ has a positive $(d-1)$-dimensional surface measure,
which implies that
$$
    \alpha(\Omega) :=  d\int_{\mathbb S_k} |\omega\cdot\Omega|^2 M_0 \ud\omega
$$
takes only positive values. By rotational symmetry it only depends on the angle between $\Omega$ and $k/|k|$ (note that $\mathbb S_k$ only depends on $k/|k|$). 
With these notations we have
$$
  \left|\Omega^{tr}(\Id - \mu A_0)\Omega \right| \ge 1 - \frac{\alpha(\Omega)}{((1-\delta)^2 + \gamma^2/2)^{1/2}} - \frac{1-\alpha(\Omega)}{1-\delta} \,.
$$
For $\delta = 0$ the right hand side takes the value
$$
   \alpha(\Omega) \left( 1 - \frac{1}{(1+\gamma^2/2)^{1/2}}\right) \ge \alpha_{min} \left( 1 - \frac{1}{(1+\gamma^2/2)^{1/2}}\right) >0 \,,
$$
where $\alpha_{min} := \min_{|\Omega|=1} \alpha(\Omega)>0$, since it is the minimum of a positive valued continuous function over a compact manifold.
It is independent of $k$ by rotational symmetry. By continuity the positivity is preserved for small enough positive $\delta$, which completes the proof.
\end{proof}

With $\mu,\J,z,k$ as in Lemma \ref{lem:stabnontriv} the system \eqref{eq:tilde} can be reduced to the scalar equation
\begin{equation}\label{tilde-rho}
   \left( 1 - a_\J - \mu\, \bar b_\J^{tr} (\Id - \mu A_\J)^{-1} b_\J \right) \tilde\rho_k = r_{\rho,k} +  \mu\, \bar b_{\mathcal J}^{tr}\,(\Id - \mu A_\J)^{-1} r_{J,k} \,.
\end{equation}
As in the proof of Lemma \ref{lem:stabnontriv} the essential information will come from the case $\mu\le d$, $\J=0$. Straightforward computation gives
$$
    a_0 = c_0 \,,\qquad \bar b_0 = b_0 = c_1\frac{k}{|k|} \,,\qquad A_0 k = c_2 k \,,
$$
with
$$
   c_j(z,k) = \int_\Sd \frac{\omega_1^j M_0}{1+z+i|k|\omega_1}\ud\omega \,,\qquad j=0,1,2\,.
$$
With these observations, the coefficient of $\tilde\rho_k$ in \eqref{tilde-rho} with $\mu\le d$, $\J=0$ becomes
\begin{equation}\label{rho-coeff}
    h(z,k) := 1 - c_0 - \frac{\mu c_1^2}{1-\mu c_2} \,,
\end{equation}
and we shall need estimates on $c_0,c_1,c_2$:

\begin{lemma}\label{lem:est-cj}
Let $\Re(z)\ge 0$ and $0\ne k\in \gamma\mathbb Z^d$. Then
$$
   \Re(c_0) \le 1 - \phi_0(\gamma,d) \,,\qquad |c_1| \le \frac{1}{2\sqrt{d}} + \frac{1}{\gamma} \,,\qquad
   d|c_2| \le 1 - \alpha_2(d,\eps)\phi_2(\eps\gamma)  \,,
$$
with $\phi_0(\cdot,d),\phi_2:\,[0,\infty) \to [0,1)$ continuous and increasing, $\phi_0(0,d)=\phi_2(0)=0$, $\phi_0(\infty,d)=\phi_2(\infty)=1$, and with
$$
    \alpha_2(d,\eps) := d\int_{\mathbb S_\eps^{d-1}} \omega_1^2 M_0 \ud\omega \,,\qquad
    \mathbb S_\eps^{d-1} := \{\omega\in\Sd:\, \omega_1 \ge \eps\} \,,\quad 0<\eps<1 \,,
 $$
a continuous, strictly decreasing function of $\eps$, satisfying $\alpha_2(d,0)=1/2$, $\alpha_2(d,1)=0$.
\end{lemma}

\begin{proof}
In the following we set $z=x+iy$, $x\ge 0$, and compute
\begin{align}
  \Re(c_0) &= \int_{\Sd} \frac{(1+x)M_0 \ud\omega}{(1+x)^2 + (y+|k|\omega_1)^2} < \frac{1}{1+x} \le 1\,.
\end{align}
On the other hand, for $d\ge 3$,
\begin{align}\label{Re-est}
  \Re(c_0)  \le c_d \int_{-1}^1 \frac{(1+x)d\omega_1}{(1+x)^2 + (y+|k|\omega_1)^2} = \frac{c_d}{|k|} \int_{(y-|k|)/(1+x)}^{(y+|k|)/(1+x)} \frac{du}{1+u^2} \le \frac{c_d \pi}{|k|} \,.
\end{align}
For $d=2$ and $|k|>1$,
$$
   \Re(c_0)  = \frac{1}{2\pi}\int_{-1}^1 \frac{(1+x)}{(1+x)^2 + (y+|k|\omega_1)^2} \frac{d\omega_1}{\sqrt{1-\omega_1^2}} = \Re_1 + \Re_2
$$
with
$$
   \Re_1 := \frac{1}{2\pi}\int_{(1-\omega_1^2)|k|<1} \frac{(1+x)}{(1+x)^2 + (y+|k|\omega_1)^2} \frac{d\omega_1}{\sqrt{1-\omega_1^2}} 
   \le \frac{1}{\pi(1+x)}\int_{\sqrt{1-1/|k|}}^1 \frac{d\omega_1}{\sqrt{1-\omega_1}} \le \frac{2}{\pi\sqrt{|k|}}
$$
and
$$
  \Re_2 \le \frac{\sqrt{|k|}}{2\pi}  \int_{-1}^1 \frac{(1+x)d\omega_1}{(1+x)^2 + (y+|k|\omega_1)^2} \le \frac{1}{2\sqrt{|k|}} \,,
$$
where the last inequality follows from \eqref{Re-est}.
This completes the proof of the estimate for $\Re(c_0)$ with $\phi_0(\gamma,d) = \max\{0,1-c_d\pi/\gamma\}$ for $d\ge 3$ and a similar definition for $d=2$.\\
For the estimation of $c_1$ we start rewriting it by symmetrization ($\omega_1 \to -\omega_1$):
\begin{align}
  c_1 &= \frac{1}{2} \int_{\Sd} \left( \frac{1}{1+ z + i|k|\omega_1} - \frac{1}{1+ z - i|k|\omega_1}\right)\omega_1 M_0 \ud\omega
     = -i|k| \int_{\Sd} \frac{\omega_1^2 M_0 \ud\omega}{(1+z)^2 + |k|^2\omega_1^2} 
\end{align}
This implies
$$
   |c_1| \le |k| \int_{\Sd} \frac{\omega_1^2 M_0 \ud\omega}{\sqrt{((1+x)^2-y^2+ |k|^2\omega_1^2)^2 + 4(1+x)^2 y^2}} \,.
$$
Minimization of the denominator with respect to $x\ge0$ and $y\in\mathbb R$ gives
$$
   \sqrt{((1+x)^2-y^2+ |k|^2\omega_1^2)^2 + 4(1+x)^2 y^2} \ge \left\{\begin{array}{ll} 1+|k|^2\omega_1^2 & \mbox{for }  |k|^2\omega_1^2 \le 1 \,,\\
                                                                                                              2|k| |\omega_1| & \mbox{for }  |k|^2\omega_1^2 > 1 \,,\end{array}\right.
$$
and, thus,
$$
  |c_1| \le \frac{1}{2} \int_{|k|^2\omega_1^2 > 1} |\omega_1| M_0 \ud\omega + |k| \int_{|k|^2\omega_1^2 \le 1} \frac{\omega_1^2 M_0 \ud\omega}{1+ |k|^2\omega_1^2}
    \le \frac{1}{2\sqrt{d}} + \frac{1}{|k|} \,.
$$
Jensen's inequality has been used for the estimation of the first term.\\
For $c_2$ we proceed similarly to the estimation of $\Re(c_0)$:
\begin{align}
  d|c_2| &\le d\int_{\Sd} \frac{\omega_1^2 M_0 \ud\omega}{\sqrt{(1+x)^2 + (y + |k|\omega_1)^2}} \le 
  d\int_{\mathbb S_\eps^{d-1}} \frac{\omega_1^2 M_0 \ud\omega}{\sqrt{1 + \eps^2\gamma^2}} + d\int_{\Sd\setminus \mathbb S_\eps^{d-1}} \omega_1^2 M_0 \ud\omega \\
  &= \frac{\alpha_2}{\sqrt{1+\eps^2\gamma^2}} + 1 - \alpha_2 \,.
\end{align}
Here we have assumed $y\ge 0$, which can always be achieved by the transformation $\omega\to-\omega$. This completes the proof with $\phi_2(u) = 1 - (1+u^2)^{-1/2}$.
The stated properties of $\alpha_2$ are obvious from its definition. 
\end{proof}

\begin{corollary}\label{cor:coeff}
There exist $\gamma_{min}, \delta, \kappa>0$ with $\delta <1$, such that for $0<\mu\le d+\kappa$, $\J \in \mathcal S(\mu)$, $\Re(z)\ge -\delta$, $\gamma\ge \gamma_{min}$, 
$0\ne k\in \gamma\mathbb Z^d$, 
$$
    \Re\left( 1 - a_\J - \mu\, \bar b_\J^{tr} (\Id - \mu A_\J)^{-1} b_\J\right) \ge \frac{1}{5} \,,
$$
holds, where $a_\J$, $b_\J$, $\bar b_\J$, and $A_\J$ are defined in \eqref{aJ-etal}.
\end{corollary}

\begin{proof}
With the estimates of the lemma, we get for the coefficient \eqref{rho-coeff} for $\Re(z)\ge 0$ and  $0<\mu\leq d$:
\begin{align}
    &\Re\left( h\right) \ge 1 - \Re(c_0) - \frac{\mu |c_1|^2}{1-\mu |c_2|} \ge 1 - \Re(c_0) - \frac{d |c_1|^2}{1-d |c_2|} \\
    &\ge \phi_0(\gamma,d) - \frac{1}{4\alpha_2(d,\eps)\phi_2(\eps\gamma)}\left( 1 + \frac{2\sqrt{d}}{\gamma} \right)^2 \ge \frac{1}{4}\,,
\end{align}
where the last inequality is achieved by first choosing $\eps>0$ small enough to get $\alpha_2(d,\eps) = 3/8$ and then choosing $\gamma$ large enough, i.e. 
$\gamma\ge \gamma_{min}(d)>0$.
By continuity with respect to $z$, $\mu$, and $\J$, this result can be extended to $0<-\Re(z)\ll 1$ and $0<\mu-d\ll 1$.
\end{proof}

By Lemma \ref{lem:stabnontriv} and Corollary \ref{cor:coeff} the system \eqref{eq:tilde} can be solved, and we need estimates of the solution.

\begin{lemma}\label{lem:FL-est}
Let $0<\mu\le d+\kappa$, $\J \in \mathcal S(\mu)$, $\Re(z)\ge -\delta$, $\gamma\ge \gamma_{min}$, 
$0\ne k\in \gamma\mathbb Z^d$ (with the notation of Corollary \ref{cor:coeff}). Then \eqref{eq:tilde} has a unique solution, and there exists $C>0$, independent from $z$, $k$, and $f^\circ$,
such that
    \begin{align} \label{est:ylarge} 
        |\tilde{\rho}_k(z)|+|\tilde{J}_k(z)|&\leq %\frac{C}{\langle y\rangle } \|\hat{f}^\circ_k(\cdot)\|_{L^1(\Sd)}\leq 
             \frac{C}{\langle \Im(z)\rangle } \|\hat{f}^\circ_k\|_{L^2(\Sd)}  &&\text{for $|\Im(z)|\geq 2 |k|$} \,,\mbox{where } \langle y\rangle := \sqrt{1+|y|^2} \,,\\
        |\tilde{\rho}_k(z)|+|\tilde{J}_k(z)|&\leq \frac{C}{\langle k\rangle^{1/5} }\|\hat{f}^\circ_k\|_{L^2(\Sd)}  &&\text{for $|\Im(z)|\leq 2 |k|$} \,. \label{est:ysmall}
    \end{align}
\end{lemma}

\begin{proof}
Lemma \ref{lem:stabnontriv} and Corollary \ref{cor:coeff} imply the bounded invertibility of the coefficient matrix, i.e.,
$$
  |\tilde{\rho}_k(z)|+|\tilde{J}_k(z)| \leq C ( |r_{\rho,k}(z)|+|r_{J,k}(z)|) \,. 
$$
With $z=x+iy$, in the case $|y|\geq 2 |k|$ we estimate
    \begin{align*}
        |r_{\rho,k}(z)|+|r_{J,k}(z)|&\leq \left| \int_{\Sd} \frac{\omega \hat{f}^\circ_k(\omega)}{1+z+ik\cdot \omega } \ud{\omega} \right| +\left| \int_{\Sd} \frac{ \hat{f}^\circ_k(\omega)}{1+z+ik\cdot \omega  } \ud{\omega} \right| \\
        &\leq 2 \int_{\Sd} \frac{| \hat{f}^\circ_k(\omega)|}{|1+x+i(y + k\cdot \omega)  |} \ud{\omega}  \,,
    \end{align*}
    and observe that the condition ensures $ |1+x+ i(y+k\cdot \omega)|\geq c \langle y\rangle $ with a constant $c>0$ independent of $k\in \mathbb{Z}^d$ and $y\in \Reals$. Inserting this above yields  
    \begin{align*}
        |r_{\rho,k}(z)|+|r_{J,k}(z)|&\leq \frac{2}{c\langle y\rangle }  \int_{\Sd}|\hat{f}^\circ_k(\omega)| \ud{\omega} \le \frac{C}{\langle y\rangle } \|\hat{f}^\circ_k\|_{L^2(\Sd)} \,,
    \end{align*}
    by the Cauchy-Schwarz inequality.
    For $|y|\leq 2|k|$ we use again the Cauchy-Schwarz inequality:
    \begin{align}
        |r_{\rho,k}(z)|+|r_{J,k}(z)| 
        &\leq 2 \|\hat{f}^\circ_k(\cdot)\|_{L^2(\Sd)} \left\|\frac{1}{1+z+ik\cdot \omega} \right\|_{L^2(\Sd)} \,.
    \end{align}
    It remains to estimate the norm on the right. By rotational symmetry we may assume $k=|k|e_1$ and compute
    \begin{align}\label{eq:L^2prep}
        \left\|\frac{1}{1+z+ik\cdot \omega} \right\|^2_{L^2(\Sd)} &= \int_{\Sd} \frac{1}{(1+x)^2+(y+|k| \omega_1)^2} \ud{\omega} \,.
    \end{align}
The set 
    \begin{align}\label{A_alpha}
        A_\alpha (y,k)          &:= \{\omega\in \Sd: |y+ |k|\omega_1|\leq |k|^{1-\alpha}\} \,,\qquad \alpha>0 \,, 
    \end{align}
is the intersection of the sphere with a strip of width $|k|^{-\alpha}$. Therefore it satisfies
    \begin{align}\label{av-prop}
         |A_\alpha (y,k)|\leq C  \langle k\rangle^{-\alpha/2} \,. 
    \end{align}
    We split the integral~\eqref{eq:L^2prep} as
    \begin{align*}
        \left\|\frac{1}{1+z+ik\cdot \omega} \right\|^2_{L^2(\Sd)} &= \int_{A_\alpha} \frac{1}{(1+x)^2+(y+|k| \omega_1)^2} \ud{\omega} 
        +\int_{\Sd \setminus A_\alpha} \frac{1}{(1+x)^2+(y+|k| \omega_1)^2} \ud{\omega} \\
        &\leq  C \left(\langle k\rangle^{-\alpha/2} + \langle k\rangle^{2\alpha-2} \right) \,.
    \end{align*} 
    Now choosing $\alpha = 4/5$ yields the claim.     
\end{proof}

The next step is to apply the inverse Laplace transform to the solution 
$$
   \tilde{f}_k(z,\omega)  = \frac{\tilde{\rho}_k(z) M_\J + \mu \tilde{J}_k(z)\cdot \nabla_J M_\J + \hat{f}_k^\circ(\omega)}{1+ z+i k\cdot \omega} \,,
$$
of \eqref{eq:F-L}, where the integration contour is chosen to be $z=x+iy \in -\delta + i\mathbb R$, with $\delta<1$ as in Corollary \ref{cor:coeff}:
\begin{align}
   \hat f_k(t,\omega) &= \frac{e^{-\delta t}}{2\pi} \lim_{M\to\infty} \int_{-M}^M e^{iyt} \tilde f_k(z,\omega) \ud y \\
   &= \frac{e^{-\delta t}}{2\pi} \lim_{M\to\infty} \int_{-M}^M e^{iyt} \frac{\tilde{\rho}_k M_\J + \mu \tilde{J}_k\cdot \nabla_J M_\J}{1+z +i k\cdot \omega} \ud y
     + e^{-t(1+ik\cdot\omega)} \hat f_k^\circ \,.
\end{align}
We split the integral as 
\begin{align}
    \hat{f}_k(t,\omega) =& \frac{e^{-\delta t}}{2\pi} \lim_{M\rightarrow \infty} \int_{2|k|\le |y| \le M} e^{iyt} \frac{\tilde{\rho}_k M_\J+\mu \tilde{J}_k \nabla_\J M_\J }{1+z+ik\cdot \omega } \ud{y} \\
    +&\frac{e^{-\delta t}}{2\pi}\int_{|y|\le2|k|} e^{iyt} \frac{\tilde{\rho}_k M_\J+\mu \tilde{J}_k \nabla_\J M_\J }{1+z+ik\cdot \omega } \ud{y}+ e^{-t(1+ik\cdot \omega)} \hat{f}_k^\circ .
\end{align}
For the first integral we use~\eqref{est:ylarge} and for the second integral~\eqref{est:ysmall} to obtain
\begin{align}
    |\hat{f}_k(t,\omega)| \leq C e^{-\delta t} \left(\lim_{M\rightarrow \infty }\int_{-M}^M \frac{\|\hat{f}^\circ_k(\cdot)\|_{L^2(\Sd)}}{\langle y\rangle^2} \ud{y} + \int_{-2|k|}^{2|k|} \frac{\|\hat{f}^\circ_k(\cdot)\|_{L^2(\Sd)}}{|1-\delta+i (y+k\cdot \omega)|\langle k\rangle^{1/5}} \ud{y} + |\hat{f}^\circ_k| \right).
\end{align}
The second integral is computed explicitly:
\begin{align}
    \int_{-2|k|}^{2|k|} \frac{1}{|1-\delta+i (y+k\cdot \omega)|\langle k\rangle^{1/5}} \ud{y} 
    & = \langle k\rangle^{-1/5} \left( \sinh^{-1}\left(\frac{k\cdot\omega+ 2|k|}{1-\delta}\right) - \sinh^{-1}\left(\frac{k\cdot\omega- 2|k|}{1-\delta}\right)\right) \\
    &\le 2 \langle k\rangle^{-1/5} \sinh^{-1}\left(\frac{3|k|}{1-\delta}\right) \le C \,,
\end{align}
with $C$ independent from $k$. Inserting this estimate above yields 
\begin{align}
    |\hat{f}_k(t,\omega)| \leq C e^{-\delta t} \left( \|\hat{f}^\circ_k\|_{L^2(\Sd)}+|\hat{f}^\circ_k(\omega)| \right) \,,
\end{align}
implying
\begin{equation}\label{hatf-est}
   \|\hat f_k\|_{L^2(\Sd)} \le C e^{-\delta t} \|\hat f_k^\circ\|_{L^2(\Sd)} \,.
\end{equation}
It remains to analyze the case $k=0$. The discussion has already been started in Section \ref{sec:results_linearized}, since $\hat f_0 = \overline f$, satisfying
\begin{equation}\label{eq:barf}
   \partial_t \overline f = \mu\overline{J_f}\cdot\nabla_J M_\J - \overline f \,,
\end{equation}
where $\overline{J_f}$ solves the linear ODE system \eqref{Jf-ODE}, and we collect results on the properties of its coefficient matrix.

\begin{lemma}\label{lem:CJ}
Let $\mu>0$, $\J\in\mathcal{S}(\mu)$, and let
$$
   C_{\J} : = \mu \int_\Sd \omega\otimes\nabla_J M_{\J}\ud\omega - \Id \,.
$$
Then $C_\J$ is symmetric and for $\mu\le d$ (i.e. $\J=0$) it is given by
\begin{equation}\label{C0}
  C_0 = \left(\frac{\mu}{d}-1\right)\Id \,.
\end{equation}
For $\mu>d$ (i.e. $|\J|= L(\mu)>0$ with $L(\mu)$ satisfying \eqref{eq:consistency_relation2}), there is a $(d-1)$-dimensional nullspace given by $\mathcal{N}(C_\J) = \J^\bot$.
The remaining eigenvalue $\lambda_\J$ (with eigenvector $\J$) satisfies
$$
   \lambda_\J = \mu-d - \frac{|\J|^2}{\mu} = - 2\left(\frac{\mu}{d}-1\right) + O((\mu-d)^2) \,,\qquad\mbox{as } \mu\to d+\,.
$$
\end{lemma}

\begin{proof}
The gradient with respect to $J$ of the von Mises distribution evaluated at an equilibrium flux $\J\in\mathcal{S}(\mu)$ is given by
\begin{equation}\label{grad-MJ}
  \nabla_J M_\J(\omega) = \omega M_\J(\omega) - \frac{e^{\omega\cdot\J}}{Z(\J)^2} \int_\Sd \omega' e^{\omega'\cdot\J}\ud\omega'
   = \left( \omega - \frac{\J}{\mu}\right)M_\J(\omega) \,,
\end{equation}
where the consistency relation \eqref{eq:consistency} has been used. This implies
$$
    C_\J = \mu \int_\Sd \omega\otimes\omega M_{\J}\ud\omega - \frac{1}{\mu}\J\otimes\J - \Id \,,
$$
and, thus, the symmetry of $C_\J$. Since $M_0 = |\Sd|^{-1}$, \eqref{C0} follows immediately, since
$$
    \int_\Sd \omega_i \omega_j \ud\omega = 0\quad\mbox{for }i\ne j \,,\qquad \mbox{and}\qquad \int_\Sd \omega_i^2 \ud\omega = \frac{|\Sd|}{d} \,.
$$
For $\mu> d$, taking the gradient of the equilibrium relation \eqref{eq:consistency} along $\me(\mu)$ shows $\mathcal{N}(C_\J) = \J^\bot$. With the representation
$\omega = \omega_1 \frac{\J}{|\J|} + \omega^\bot$ (whence $M_\J$ becomes a function of $\omega_1$) the computation
$$
   \left(\int_\Sd \omega\otimes\omega M_{\J}\ud\omega  \right) \J = \int_\Sd \left( \omega_1 \frac{\J}{|\J|} + \omega^\bot\right) \omega_1 |\J| M_\J \ud\omega
   = \left( \int_\Sd \omega_1^2 M_\J \ud\omega \right)\J
$$
shows that $\J$ is an eigenvector of $C_\J$. For the computation of the eigenvalue, the value of the integral on the right hand side is needed. \\
We choose a unit vector $z\in \mathcal{N}(C_\J)$, i.e. $z\cdot\J = 0$ and $|z|=1$. Then, with the representation $\omega = \omega_1 \frac{\J}{|\J|} + \omega_2 z + \omega^\bot$
we have
$$
   0 = C_\J z = \mu \int_\Sd \left( \omega_1 \frac{\J}{|\J|} + \omega_2 z+ \omega^\bot\right) \omega_2 M_\J \ud\omega - z 
   = \left( \mu \int_\Sd \omega_2^2 M_\J \ud\omega - 1\right)z \,.
$$
This implies
$$
    \int_\Sd \omega_2^2 M_\J \ud\omega = \frac{1}{\mu} \,.
$$
By symmetry this remains valid when $\omega_2$ is replaced by $\omega_j$, $j\ne 1$. Since $M_\J$ is normalized, we also have
$$
   1 = \int_\Sd M_\J \ud\omega = \int_\Sd \omega_1^2 M_\J \ud\omega + (d-1) \int_\Sd \omega_2^2 M_\J \ud\omega \,,
$$
Combining these results leads to
$$
    \int_\Sd \omega_1^2 M_\J \ud\omega = 1 - \frac{d-1}{\mu} \,,
$$
which proves the formula for $\lambda_\J$. The asymptotic expansion is a direct consequence of \eqref{L-as}.
\end{proof}

The lemma implies that $\overline{J_f}$ tends to its equilibrium exponentially for $\mu\ne d$ with linear degeneration of the decay rate as $\mu\to d$.
The equilibrium is zero for $\mu< d$, and it is given by $P_\J^\bot \overline{J_{f^\circ}}$ for $\mu>d$. It is obvious from \eqref{eq:barf} that $\overline f = \hat f_0$
also converges exponentially for $\mu\ne d$. Combination of these observations with \eqref{hatf-est} and application of Parseval's identity complete the
proof of Theorem \ref{th:spec-stab}.

\section{Nonlinear stability for small data -- proof of Lemma \ref{lem:existence}  and Theorem~\ref{th:convergence}} 
\label{sec:convergenceProof}

We consider an initial datum close to the manifold of equilibria, i.e. 
$$
   \|F^\circ-F_{\mu,J_1}\|_{H^m_x} =: \eps
$$
is assumed to be small ($\eps<\eps_0$), with $J_1\in \mathcal{S}(\mu)$ and $H^m_x = H^m_x(\T^d \times \Sd)$ as in Theorem \ref{th:spec-stab}.
We start by choosing an improved estimate of the equilibrium flux by projecting the average flux of the initial datum to the equilibrium manifold $\mathcal{S}(\mu)$:
$$
    J_2 := 0 \quad\mbox{for } \mu<d \,,\qquad \mbox{and}\qquad J_2 := L_\mu\frac{\overline{J_{F^\circ}}}{\left|\overline{J_{F^\circ}}\right|} \quad\mbox{for } \mu>d \,,
$$
with $L_\mu$ as in Definition \ref{def:Steady}. In the second case $J_2$ is well defined, if we assume $\eps_0|\Sp^{d-1}|^{1/2}\le L_\mu$, since
$$
    \left|\overline{J_{F^\circ}}\right| \ge L_\mu - \left|\left|\overline{J_{F^\circ}}\right| -\left| J_1\right|\right| 
    \ge L_\mu - \left| \overline{J_{F^\circ}} - J_1\right| \ge L_\mu - |\Sp^{d-1}|^{1/2}\|F^\circ-F_{\mu,J_1}\|_{L^2} \ge L_\mu - |\Sp^{d-1}|^{1/2}\eps>0\,.
$$
As a consequence of the definition of $J_2$ we have
\begin{align}\label{J-dist}
   |J_2 - \overline{J_{F^\circ}}| \le |J_1 - \overline{J_{F^\circ}}| \le \eps \,,
\end{align}
and 
\begin{align}\label{f0-est}
       \|F^\circ - F_{\mu,J_2}\|_{H^m_x} \leq \|F^\circ - F_{\mu,J_1}\|_{H^m_x} +\|F_{\mu, J_1} - F_{\mu,J_2}\|_{H^m_x}
        \leq \eps + C |J_1-J_2| \le (1+2C)\eps \,,
 \end{align}
where we have used that the mapping $J\mapsto  F_{\mu,J}$ from $ \Reals^d$ to $H^m_x$ is locally Lipschitz.

The next step is linearization around $F_{\mu,J_2}$. We rewrite the Vicsek-BGK problem \eqref{nonlinearBGK} in terms of  $f_1=F-F_{\mu,J_2}$:
   \begin{equation}\label{eq:linearized_around_F2}
   \begin{aligned} 
       &\partial_t f_1  = \La_{\mu,J_2} f_1 + \mathcal S_{\mu,J_2}(f_1) \,,\\
       &f_1(t=0)= f^\circ := F^\circ- F_{\mu, J_2} \,,
   \end{aligned}
   \end{equation}
where  
   \begin{align}
      \La_{\mu,J_2}f &= \rho_f M_{J_2} + \mu  J_f\cdot  \nabla_J M_{J_2}- f - \gamma \omega \cdot \nabla_x f    \qquad\mbox{and} \\
      \mathcal S_{\mu,J_2}(f) &=  (\mu+\rho_f)(M_{J_2+J_f} - M_{J_2})  - \mu J_f\cdot \nabla_J M_{J_2}
   \end{align}
are the linearized operator (as in \eqref{eq:linearizedBGK}) and, respectively, the nonlinear remainder.
With the semigroup $\mathcal T_{\mu,J_2}$ generated by $\La_{\mu,J_2}$,  the Duhamel formulation of \eqref{eq:linearized_around_F2} reads
   \begin{align}\label{eq:Duhamel}
       f_1(t) = \mathcal T_{\mu,J_2}(t) f^\circ + \int_0^t \mathcal T_{\mu,J_2}(t-s) \mathcal S_{\mu,J_2}(f_1(s)) \ud s \,.
   \end{align}
Theorem \ref{th:spec-stab} shows that $\mathcal T_{\mu,J_2}$ is bounded on $H^m_x$. Since $\mathcal S_{\mu,J_2}(f)$ depends smoothly on $\rho_f$ and $J_f$
and since $H^m(\T^d)$ is a Banach algebra (as a consequence of the continuous imbedding $H^m(\T^d)\hookrightarrow L^\infty(\T^d)$), local existence of $f_1$ in $H^m_x$ 
can be shown by Picard iteration. Standard arguments prove that finite time blow-up implies blow-up of $\|\rho_{f_1}\|_{L^\infty(\T^d)}$, completing the proof of Lemma \ref{lem:existence}.

The Banach algebra property also implies that for each $R>0$ there exists $C_R>0$, such that
   \begin{align} \label{eq:estimate_S}
       \|\mathcal S_{\mu,J_2}(f)\|_{H^m_x} \leq C_R\|f\|^2_{H^m_x} \,,\qquad \mbox{for } \|f\|_{H^m_x} \le R \,.
   \end{align}
Since by construction $\int f^\circ dxd\omega=P^\bot_{J_2} \overline{J_{f^\circ}}= 0$, Theorem~ \ref{th:spec-stab} can be used for the action of the semigroup on $f^\circ$:
    \begin{align} \label{eq:decay_constant_linear}
        \|\mathcal T_{\mu,J_2}(t) f^\circ\|_{H^m_x} \leq C e^{-\lambda t} \| f^\circ\|_{H^m_x} \le C_1\eps e^{-\lambda t}\,,
    \end{align}
by using \eqref{f0-est}. Application of these results to \eqref{eq:Duhamel} gives
$$
   \|f_1(t)\|_{H^m_x} \le C_1\eps e^{-\lambda t} + C_R \int_0^t \|f_1(s)\|^2_{H^m_x} \ud s \,,
$$
as long as $\|f_1(s)\|_{H^m_x}\le R$ for $0\le s\le t$. For fixed $R$ we restrict $\eps_0$ further to satisfy $2C_1\eps_0 \le R$ and try to prove
\begin{align}\label{f(t)-est}
   \|f_1(t)\|_{H^m_x} \le 2 C_1\eps e^{-\lambda t}
\end{align}
on an appropriate time interval. This is satisfied, if
$$
   4C_R C_1^2\eps^2 \int_0^t e^{-2\lambda s}\ud s \le \frac{2C_R C_1^2 \eps^2}{\lambda} \le C_1\eps e^{-\lambda t} \,,
$$
which holds for 
\begin{align}\label{est-int}
  t\le \frac{1}{\lambda} \log\left(\frac{\lambda}{2C_R C_1 \eps}\right) \,.
\end{align}
Now we set 
\begin{align}\label{T-def}
    T := \frac{1}{\lambda} \log\left(4C_1\right) \,,
\end{align}
which satisfies \eqref{est-int} under the further restriction
$$
    \eps_0 \le \frac{\lambda}{8C_R C_1^2} \,.
$$
This implies  
$$
   \|f_1\|_{H^m_x} \le 2C_1\eps \quad\mbox{on } [0,T] \qquad\mbox{and}\qquad  \|f_1(T)\|_{H^m_x} \le \frac{\eps}{2} \,.
$$
Now the procedure is iterated with $\eps$ replaced by $\eps/2$, producing a sequence $(J_{k+1},f_k)_{k\ge 1}$ with $J_{k+1}\in S(\mu)$ and 
$f_k= F-F_{\mu,J_{k+1}}: \,[(k-1)T,kT]\to H^m_x$, such that 
$P^\bot_{J_{k+1}} \overline{J_{f_k((k-1)T)}}=0$,
$$
   \|f_k\|_{H^m_x} \le \frac{C_1 \eps}{2^{k-2}} \quad\mbox{on } [(k-1)T,kT] \,,\qquad\mbox{and}\qquad  \|f_k(kT)\|_{H^m_x} \le \frac{\eps}{2^k} \,.
$$
From \eqref{J-dist} we deduce $|J_1-J_2|\le 2\eps$ and therefore
\begin{align}\label{Jk-Jk+1}
    |J_k-J_{k+1}| \le \frac{\eps}{2^{k-2}} \,,\qquad k\ge 1 \,.
\end{align}
This implies convergence: 
$$
    \lim_{k\to\infty} J_k = J_\infty \,.
$$
We expect $F(t) \to F_{\mu,J_\infty}$ as $t\to\infty$ and estimate for $(k-1)T\le t \le kT$:
\begin{align}
  \|F(t) - F_{\mu,J_\infty}\|_{H^m_x} &\le \|f_k(t)\|_{H^m_x} + \sum_{l=k+1}^\infty \|F_{\mu,J_l} - F_{\mu,J_{l+1}}\|_{H^m_x} \le \frac{C_1\eps}{2^{k-2}} 
  + \sum_{l=k+1}^\infty \frac{C_2 \eps}{2^l} = \frac{C_3\eps}{2^k}  \\
  &\le \frac{C_3\eps}{2^{t/T}} = C_3 \eps e^{-\hat\lambda t}\,,\qquad \mbox{with } \hat\lambda = \frac{\lambda \log 2}{\log(4C_1)} \, ,
\end{align}
where we have used Lipschitz continuity as in \eqref{f0-est} as well as \eqref{Jk-Jk+1} and \eqref{T-def}. This completes the proof.

\section{Global existence of weak solutions -- proof of Theorem~\ref{th:existence_solution}  }
\label{sec:existence}

In this section we prove the global well-posedness of the nonlinear equation, proceeding similarly to the methods of~\cite{perthame1989global} for the Boltzmann-BGK equation. We introduce a family of approximate equations
\begin{equation} \label{eq:approxBGK}
\begin{aligned}
    \partial_t F + \omega \cdot \nabla_x F &= \rho_F M_{J^\epsilon_F} - F \,, \\
        F(t=0) &= F^\circ \,,
\end{aligned}
\end{equation}
where $J_F^\eps$ is given by
\begin{align} \label{def:Jeps}
    J^\eps_F = \frac{J_F}{|J_F|} \min\{|J_F|, \frac{1}{\eps}  \} \,,\qquad \eps > 0\,.
\end{align}
For the initial datum, we impose boundedness of the total mass and of the logarithmic entropy:
\begin{align}
    \int_{\T^d} \int_{S^{d-1}} F^\circ (1+|\log F^\circ|) \ud{x} \ud{\omega} < \infty \,.  
\end{align}
For every $\eps>0$, classical methods show that the problem~\eqref{eq:approxBGK} has a unique global mild solution solution in $C([0,\infty); L^1(\T^d\times\Sp^{d-1}))$. In order to pass to the limit $\eps\rightarrow 0$, we need estimates which are uniform in $\eps\rightarrow 0$. For this purpose, a bound for the $J$-dependence of the von Mises distribution is needed.

\begin{lemma} \label{lem:bounds_Mises}
   For $d\ge 2$ there exists $C_d>0$ such that the von Mises distribution \eqref{def:MJ} satisfies
    \begin{align}
        0 < M_J(\omega) \leq C_d \left(1+ |J|^{(d-1)/2} \right) \,,\qquad \forall\,\omega\in \Sp^{d-1}\,.
    \end{align}
\end{lemma}
\begin{proof}
The von Mises distribution can be written as 
\begin{align}
   M_J(\omega) = c_d \,e^{\omega\cdot J} \left(\int_0^\pi e^{|J|\cos\theta} \sin^{d-2}\theta \,d\theta\right)^{-1} \,.
\end{align}
Its maximum
\begin{align}
   \sup_{\omega} M_J(\omega) = c_d \left(\int_0^\pi e^{|J|(\cos\theta-1)} \sin^{d-2}\theta \,d\theta\right)^{-1} \,,
\end{align}
is a smooth function of $|J|\in [0,\infty)$. Its behaviour for large $|J|$ can be studied by the new coordinate $u = |J|(1-\cos\theta)$, giving
\begin{align}
  \int_0^\pi e^{|J|(\cos\theta-1)} \sin^{d-2}\theta \,d\theta = \frac{1}{|J|} \int_0^{2|J|} e^{-u} \left( \frac{2u}{|J|} - \frac{u^2}{|J|^2}\right)^{(d-3)/2} du
  \approx |J|^{-\frac{d-1}{2}} 2^{\frac{d-3}{2}} \Gamma\left(\frac{d-1}{2}\right) \,,
\end{align}
as $|J|\to\infty$.
\end{proof}

\begin{lemma} \label{lem:epsuniform}
Let the assumptions of Theorem \ref{th:existence_solution} hold. Then there exist $c,C>0$, independent of $\eps>0$, such that the solution 
$F_\eps$ of \eqref{eq:approxBGK} satisfies
    \begin{align}
        \int_{\T^d} \int_{\Sd} F_\eps(t,x,\omega) (1+|\log F_\eps(t,x,\omega)|) \ud{x} \ud{\omega} \le C(1+e^{ct})\,.
    \end{align}
\end{lemma}

\begin{proof}
By the mass conservation and by the boundedness of the domain, the integral of $F_\eps$ and the product of $F_\eps$ with the negative part of the logarithm are uniformly
bounded. Therefore it suffices to show uniform boundedness of the entropy functional
$$\mathcal{E}[F_\eps]=\int_{\T^d} \int_{\Sd} \left(\frac{1}{e}+F_\eps\log F_\eps \right) \ud x\ud\omega \ge 0\,,$$
with the time derivative
\begin{align}
 & \frac{\ud}{\ud t}\mathcal{E}[F_\eps]  = \int_{\T^d} \int_{\Sd} ( \rho_{F_\eps} M_{J_{F_\eps}^\eps}-F_\eps) \log F_\eps\, \ud x \ud \omega \nonumber\\
   &= \frac{|\Sp^{d-1}|}{e} - \mathcal{E}[F_\eps] - \int_{\T^d} \int_{\Sd} \rho_{F_\eps} M_{J_{F_\eps}^\eps}\log\lp \frac{\rho_{F_\eps} M_{J^\eps_{F_\eps}}}{F_\eps} \rp \ud x \ud \omega
    + \int_{\T^d} \int_{\Sd} \rho_{F_\eps} M_{J^\eps_{F_\eps}}\log (\rho_{F_\eps} M_{J^\eps_{F_\eps}})  \nonumber\\
  &\leq \frac{|\Sp^{d-1}|}{e}- \mathcal{E}[F_\eps] + \int_{\T^d} \int_{\Sd} \rho_{F_\eps} M_{J^\eps_{F_\eps}}\log (\rho_{F_\eps} M_{J^\eps_{F_\eps}})  \ud x \ud \omega \,, \label{eq:aux_E_estimate}
\end{align}
where the inequality is due to the nonnegativity of the relative entropy
\begin{align}
  &\int_{\T^d} \int_{\Sd} \rho_F M_{J_F^\eps}\log\lp \frac{\rho_F M_{J^\eps_F}}{F} \rp \ud x \ud \omega \\
  =& \int_{\T^d} \int_{\Sd} \left(\rho_F M_{J_F^\eps}\log\lp \frac{\rho_F M_{J^\eps_F}}{F} \rp - \rho_F M_{J_F^\eps} + F\right)\ud x \ud \omega \ge 0\,.
\end{align}
After splitting the last term in \eqref{eq:aux_E_estimate} by the functional equation of the logarithm, we use the Jensen inequality
$$
    \frac{\rho_F}{|\Sp^{d-1}|} \log \frac{\rho_F}{|\Sp^{d-1}|} \le \frac{1}{|\Sp^{d-1}|} \int_{\Sp^{d-1}} F \log F\,\ud\omega
$$
for the first part to obtain
$$
   \int_{\T^d} \int_{\Sd} \rho_{F_\eps} M_{J^\eps_{F_\eps}}\log (\rho_{F_\eps} )  \ud x \ud \omega =  \int_{\T^d} \rho_{F_\eps} \log (\rho_{F_\eps} )  \ud x 
   \le \|F^\circ\|_{L^1(\TT^d\times \Sp^{d-1})} \log |\Sp^{d-1}| - \frac{|\Sp^{d-1}|}{e} + \mathcal{E}[{F_\eps}] \,.
$$
For the second part we use Lemma \ref{lem:bounds_Mises}, as well as $|J_F^\eps|\le |J_F| \le \rho_F$:
\begin{align}
  \int_{\T^d} \int_{\Sd} \rho_{F_\eps} M_{J^\eps_{F_\eps}}\log (M_{J^\eps_{F_\eps}} )  \ud x \ud \omega 
   \le \int_{\T^d} \rho_{F_\eps} \log \left(C_d \left(1+ \left|\rho_{F_\eps}\right|^{(d-1)/2} \right) \right)  \ud x \,.
\end{align}
It is easily seen that there exists $c\ge \frac{d-1}{2}$ such that $\rho\log(1+\rho^{(d-1)/2})\le c(1+\rho\log\rho)$. Applying this and collecting our results we have
\begin{align}
  \frac{\ud}{\ud t}\mathcal{E}[F_\eps]  \le  \|F^\circ\|_{L^1(\TT^d\times \Sp^{d-1})} \log (C_d|\Sp^{d-1}|^{1+c}) + c\left(1 - \frac{|\Sp^{d-1}|}{e} + \mathcal{E}[{F_\eps}] \right)\,.
\end{align}
This differential inequality implies the desired uniform (in $\eps$) boundedness of $\mathcal{E}[{F_\eps}]$ on finite time intervals.
\end{proof}

Lemma \ref{lem:epsuniform} ensures uniform integrability of $\{F_\eps, \eps>0\}$ and therefore the weak convergence
\begin{align}
    F_\eps &\rightharpoonup F \quad \text{in } L^1((0,T) \times \T^d \times \Sd) \,,
\end{align}
for appropriate subsequences by the Dunford-Pettis theorem (as in, e.g., \cite[page 196]{perthame1989global}).
In the proof of Lemma \ref{lem:epsuniform} we have bounded the last term in \eqref{eq:aux_E_estimate} in terms of $\mathcal{E}[{F_\eps}]$, implying
uniform integrability of $\rho_{F_\eps} M_{J^\eps_{F_\eps}}$. An application of the averaging lemma \ref{lem:av-L1} implies the strong convergence
$$
      \rho_{F_\eps} \rightarrow\rho_F \quad \text{in } L^1((0,T) \times \T^d) \,,
$$
and therefore, by $|J_{F_1}-J_{F_2}| \le |\rho_{F_1}-\rho_{F_2}|$, also
$$
     J_{F_\eps} \rightarrow J_F \quad \text{in } L^1((0,T) \times \T^d) \,,
$$
again for appropriate subsequences. A further application of the Dunford-Pettis theorem gives weak convergence of $\rho_{F_\eps} M_{J^\eps_{F_\eps}}$. Its 
weak limit is given by $\rho_F M_{J_F}$ since, again taking subsequences, the convergence of $\rho_{F_\eps}$ and $J_{F_\eps}$ (and therefore also of $J^\eps_{F_\eps}$)
is pointwise almost everywhere. Since the other terms in \eqref{eq:approxBGK} are linear, we can pass to the limit.

Since $f,g:= \rho_F M_{J_F}\in L^\infty((0,T), L^1(\T^d\times \Sd))$, the weak solution is actually a mild solution. In particular, for $0\le t_1\le t_2$ we have
$$
   f(t_2,x,\omega) = e^{t_1-t_2}f(t_1,x-\omega (t_2-t_1),\omega) + \int_{t_1}^{t_2} e^{t-t_2}g(t,x-\omega(t_2-t),\omega)\ud t \,.
$$
In the limit $t_2\to t_1+$, the first term on the right hand side converges to $f(t=t_1)$ (Fr\'echet-Kolmogorow) and the second to zero in $L^1(\T^d\times \Sd)$.
This proves $F\in C([0,\infty), L^1(\T^d\times \Sd))$ and completes the proof of Theorem \ref{th:existence_solution}.

\section*{Appendix A: Asymptotics close to the bifurcation point}

We look for an expansion of the solution $L(\mu)>0$ of \eqref{eq:consistency_relation2} for $0< \mu-d \ll 1$, with $L(d) = 0$.
The equation can be written as
$$
\mu \int^\pi_{0} \cos\theta \,e^{L\cos\theta} \sin^{d-2}\theta \ud \theta = L \int^\pi_{0}  e^{L\cos\theta} \sin^{d-2}\theta \ud \theta \,,
$$
and with Taylor expansion around $L=0$ we get
$$
    \mu L I_{2,d-2} + \mu \frac{L^3}{6} I_{4,d-2} = L I_{0,d-2} + \frac{L^3}{2} I_{2,d-2} + O(L^5) \,,
$$
where
$$
    I_{k,m} := \int_0^\pi \cos^k\theta \sin^m\theta \ud\theta \,.
$$
These integrals satisfy the recursions
$$
    I_{4,m} = I_{2,m} - I_{2,m+2} \,,\qquad I_{2,m} = \frac{1}{m+1} I_{0,m+2} \,,\qquad I_{0,m} = \frac{m-1}{m} I_{0,m-2} \,,
$$
which can be shown by trigonometric identities and by integrations by parts.
Using them, all the integrals appearing in the equation above can be written as multiples of $I_{0,d-2}$. This leads to
\begin{equation}\label{L-as}
    L^2 = (d+2)(\mu-d) + O((\mu-d)^2) \,,\qquad\mbox{as } \mu\to d+\,.
\end{equation}

\section*{Appendix B: Averaging lemmas}

Although the averaging lemma we need here cannot really be considered as a new result (see, e.g., \cite{jabin2009}), its precise formulation does not seem to be readily available
in the literature. We follow along the lines of \cite{golse1986regularity}, first proving a $L^2$-averaging lemma and then use it for the extension to $L^1$.

\begin{lemma}\label{lem:L2av}
Let $\mu$ be a positive bounded measure on $\Sp^{d-1}$, satisfying for appropriate $C,\beta>0$
\begin{align}\label{av-ass}
   \esssup_{e\in\Sp^{d-1},\, y\in \R^d} \mu\left(\left\{ \omega\in\Sp^{d-1}:\, |\omega\cdot e + y|\le \eps \right\}\right) \le C\eps^\beta \,,\qquad \mbox{for all } \eps>0\,. 
\end{align}
Let $f,g \in L^2(dt \otimes dx \otimes d\mu)$, $\gamma>0$, satisfy
\begin{align}\label{av-equ}
   \partial_t f + \gamma \omega\cdot\nabla_x f + f = g \qquad \mbox{in } \R\times\T^d\times\Sp^{d-1}\,.
\end{align}
Then there exists $c>0$ such that
\beqar
  \left\|\tilde f\right\|_{H^{\beta/(2+\beta)}(dt \otimes dx)} \le c \|g\|_{L^2(dt \otimes dx \otimes d\mu)} \,,\qquad \tilde f := \int_{\Sp^{d-1}} f\,d\mu \,.
\eeqar
\end{lemma}

\begin{proof}
As in Section \ref{sec:lin} we use Fourier series with respect to $x$ (with Fourier variable $k\in \gamma\mathbb Z^d$), but the Fourier transform (instead of the Laplace transform)
with respect to $t$ (with Fourier variable $\tau$). Denoting the Fourier transform by $\mathcal F[\cdot]$ we have
\beqar
  \mathcal F\left[\tilde f\right](\tau,k) = \int_{\Sp^{d-1}} \frac{\mathcal F[g](\tau,k,\omega)}{1+i(\tau+k\cdot\omega)}d\mu(\omega) \,.
\eeqar
Now we proceed similarly to the proof of Lemma \ref{lem:FL-est} with $x=0$, $y=\tau$:
\beqar
  \left|\mathcal F\left[\tilde f\right](\tau,k)\right| \le  \|\mathcal F[g](\tau,k,\cdot)\|_{L^2(d\mu)}  \left\| \frac{1}{1+i(\tau+k\cdot\omega)} \right\|_{L^2(d\mu)} \,.
\eeqar
For $|\tau|\le 2|k|$ we consider the set $A_\alpha$ defined in \eqref{A_alpha}, satisfying $\mu(A_\alpha)\le C \langle k\rangle^{-\alpha\beta}$ by \eqref{av-ass}. Then
\begin{eqnarray*}
  \left\| \frac{1}{1+i(\tau+k\cdot\omega)} \right\|_{L^2(d\mu)}^2 &=& \int_{A_\alpha} \frac{d\mu}{1 + (\tau+k\cdot\omega)^2} 
  + \int_{\Sp^{d-1}\setminus A_\alpha} \frac{d\mu}{1 + (\tau+k\cdot\omega)^2} \\
  &\le& C_1\left( \langle k\rangle^{-\alpha\beta} + \langle k\rangle^{2\alpha-2}\right) = 2C_1 \langle k\rangle^{-2\beta/(2+\beta)} \\
   &\le& C_2 \left( 1 + |k|^{2\beta/(2+\beta)} + |\tau|^{2\beta/(2+\beta)} \right)^{-1}\,,
\end{eqnarray*}
with $\alpha = \frac{2}{2+\beta}$.
For $|\tau|\ge 2|k|$ we have $|\tau + k\cdot\omega|\ge |\tau|/2$ and therefore
$$
  \left\| \frac{1}{1+i(\tau+k\cdot\omega)} \right\|_{L^2(d\mu)}^2   \le \frac{C_3}{\langle\tau\rangle^2} \le \frac{C_4}{1+\tau^2 + |k|^2} 
  \le C_5 \left( 1 + |k|^{2\beta/(2+\beta)} + |\tau|^{2\beta/(2+\beta)} \right)^{-1} \,.
$$
Combination of these results gives 
$$
    \left( 1 + |k|^{2\beta/(2+\beta)} + |\tau|^{2\beta/(2+\beta)} \right) \left|\mathcal F\left[\tilde f\right](\tau,k)\right|^2 \le c^2  \|\mathcal F[g](\tau,k,\cdot)\|_{L^2(d\mu)}^2 \,,
$$
and an application of Parseval's and Plancherel's identities completes the proof.
\end{proof}

\begin{lemma}\label{lem:av-L1}
Let \eqref{av-ass} hold. Define the operator $T$ from $L^1(dt\otimes dx \otimes d\mu)$ into $L^1(dt\otimes dx)$
by $Tg = \tilde f$, where $f$ is the unique solution in $L^1(dt\otimes dx \otimes d\mu)$ of \eqref{av-equ}.
If $K \subset L^1(dt\otimes dx\otimes d\mu)$ is bounded and uniformly integrable, then $T(K)$ is
compact in $L^1_{loc}(dt\otimes dx)$.
\end{lemma}

We omit the proof, which is the same as the proof of Proposition 3 in \cite{golse1986regularity} with the obvious changes, in particular using Lemma \ref{lem:L2av}
instead of Theorem 1 of \cite{golse1986regularity}.

\bibliographystyle{abbrv}
\bibliography{biblio}

\end{document}